\documentclass[12pt,reqno]{amsart}
\usepackage{fullpage,amssymb,amsmath,mathrsfs,enumerate,cite}

\newtheorem{thm}{Theorem}
\newtheorem{proposition}[thm]{Proposition}
\newtheorem{lemma}[thm]{Lemma}
\newtheorem{cor}[thm]{Corollary}

\theoremstyle{definition}
\newtheorem*{defin}{Definition}
\newtheorem{rem}[thm]{Remark}
\newtheorem{question}{Question}

\newcommand\setsep{;\ }

\def\dens{\operatorname{dens}}
\def\sspan{\operatorname{span}}
\def\ker{\operatorname{Ker}}
\def\eps{\varepsilon}
\def\restriction{\!\upharpoonright}

\def\M{\mathcal{M}}

\def\F{\mathcal F}

\def\qe{\mathbb Q}
\def\en{\mathbb N}
\def\P{\mathcal P}
\def\C{\mathcal{C}}
\def\ov{\overline}
\def \rng {\operatorname{Rng}}
\def \dom {\operatorname{Dom}}

\begin{document}
\title{Rich families and elementary submodels}
\author{Marek C\'uth and Ond\v{r}ej F.K. Kalenda}
\address{Department of Mathematical Analysis, Faculty of Mathematics and Physic, Charles University,
Sokolovsk\'{a} 83, Praha 8, 186 \ 75, Czech Republic}
\address{Department of Mathematical Analysis, Faculty of Mathematics and Physic, Charles University,
Sokolovsk\'{a} 83, Praha 8, 186 \ 75, Czech Republic}
\email{marek.cuth@gmail.com}
\email{kalenda@karlin.mff.cuni.cz}
\begin{abstract} We compare two methods of proving separable reduction theorems in functional analysis
-- the method of rich families and the method of elementary submodels. We show that any result proved using rich families holds also when formulated with elementary submodels and the converse is true in spaces with fundamental minimal system an in spaces 
of density $\aleph_1$. We do not know whether the converse is true in general. 
We apply our results to show that a projectional skeleton may be without loss of generality indexed by ranges of its projections.\end{abstract}
\maketitle

\section{Introduction}

For the study of nonseparable Banach spaces, construction of a separable subspace with certain property is sometimes important. It enables us to transfer properties from smaller (separable) spaces to larger ones.

One of the important approaches is the ``separable reduction''. By a separable reduction we usually mean the possibility to extend the validity of a statement from separable spaces to the nonseparable setting without knowing the proof of the statement in the separable case. This method has been used e.g. in \cite{cuth, cuthrmoutil, fabRedukce, ioffe, moors1, moors2, tiser, zajicDif}. The proof of separable reduction theorems depends on a ``separable determination'': a statement $\phi$ concerning a nonseparable Banach space $X$ is here considered to be \emph{separably determined} if
$$\text{The statement }\phi\text{ holds in $X$ }\Longleftrightarrow \forall F\in\F:\;\text{The statement }\phi\text{ holds in $F$},$$
where $\F$ is a sufficiently large family of separable subspaces of $X$; typically, for any separable subspace of $X$ there is a bigger subspace from $\F$. Note that in the literature sometimes a statement is considered to be ``separably determined'' if only the implication from the left to the right above holds. We refer to a nice introduction from \cite{fabRedukce}, where more details about the history (and not only that) may be found. 

Although in applications one makes the final deduction using just one separable subspace, it is convenient to know that the family $\F$ is large in order to join finitely many arguments together. Hence, given a statement $\phi$, we are trying to construct a large family of separable subspaces $\F$ so that the above holds for $\phi$. There are several approaches to this. One of them is the concept of \emph{rich families} introduced in \cite{borwein} by Borwein, Moors and further used e.g. in \cite{moors1, moors2, tiser, zajicDif}.

\begin{defin}Let $X$ be a Banach space. A family $\F$ of separable subspaces of $X$ is called {\em rich} if
\begin{enumerate}[\upshape (i)]
	\item each separable subspace of $X$ is contained in an element of $\F$ and
	\item for every increasing sequence $F_i$ in $\F$, $\overline{\bigcup_{i=1}^\infty F_i}$ belongs to $\F$.
\end{enumerate}
\end{defin}

Another concept is the ``method of suitable models'' (or method of elementary submodels) used in \cite{cuth, cuthrmoutil, cuthRmoutilZeleny}. The class $\F$ then consists of spaces of the form $\ov{X\cap M}$, where $M$ is a suitable model; i.e. a set for which certain finite list of formulas is absolute and which contains some countable set given in advance. We refer to the next section and \cite{cuth}, where more details may be found.

In the present paper we investigate the relationship between these two methods. This relationship is summed up in Theorem \ref{tRichAModel} which says that a statement which is separably determined by the method of rich families is also separably determined by the method of suitable models; under certain additional conditions the converse holds as well. However, 
we do not know whether these two methods are equivalent in general. It is not clear even in case $X = \C(K)$ where $K$ is a Boolean space (i.e. zero-dimensional Hausdorff compact space); see Question \ref{q3}.

A key tool to prove the mentioned theorem is a procedure of creating a rich family of separable subspaces, where every space from the family is of the form $\ov{X\cap M}$ for a suitable model; see Theorem \ref{tRich}. It seems to be a challenging problem whether the assumptions in Theorem \ref{tRich} are really needed; see Question \ref{q2}.

We apply the main theorem to clarify and simplify the definion of a projectional skeleton. This notion was introduced by W. Kubi\'s in \cite{kubis} and further investigated in \cite{cuthCMUC, cuthSimul, ferrer, kubisSkeletonEkviv}. Let us give the original definition.

\begin{defin}A \textit{projectional skeleton} in a Banach space $X$ is a family of bounded projections $\{P_s\}_{s\in\Gamma}$ indexed by an up-directed partially ordered set $\Gamma$ satisfying the following conditions:
\begin{enumerate}[\upshape (i)]
	\item Each $P_s X$ is separable.
	\item $X = \bigcup_{s\in\Gamma}P_s X$.
	\item $s\leq t \Rightarrow P_s = P_s\circ P_t = P_t\circ P_s.$
	\item Given $s_1 < s_2 < \cdots$ in $\Gamma$, $t = \sup_{n\in\omega}s_n$ exists and $P_t X = \ov{\bigcup_{n\in\en}P_{s_n}X}$.
\end{enumerate}
Given $r\geq 1$, we say that $\{P_s\}_{s\in\Gamma}$ is an \textit{r-projectional skeleton} if it is a projectional skeleton such that $\|P_s\|\leq r$ for every $s\in\Gamma$.

We say that $\{P_s\}_{s\in\Gamma}$ is a \textit{commutative projectional skeleton} if $P_s\circ P_t = P_t\circ P_s$ for any $s,t\in\Gamma$.
\end{defin}

The notion of a projectional skeleton is an important tool to study nonseparable Banach spaces since it provides 
a decomposition of such a space to separable pieces. There is a hope that if we glue them together, their properties will be preserved by the nonseparable Banach space we started with. It is known that inductive arguments work well in a space with a projectional skeleton. Consequently, every space with a projectional skeleton has a strong Markushevich basis and an LUR renorming; see the introduction of \cite{cuthSimul} for more details.

One can ask, what is the mysterious index set $\Gamma$ in the definition of a projectional skeleton above. It is claimed already in \cite{kubisSkeletonEkviv} that we may always assume that the projectional skeleton is indexed by ranges of its projections. This statement is used also in \cite{ferrer}. However, the proof is not complete and we fill in the gap in the last section of this paper; see Theorem \ref{tSkeleton}. We give two proofs of Theorem \ref{tSkeleton} - one uses the results above and one is absolutely self-contained.\\

Let us recall the most relevant notions, definitions and notations:
We denote by $\omega$ the set of all natural numbers (including $0$), by $\en$ the set $\omega\setminus\{0\}$. Whenever we say that a set is countable, we mean that the set is either finite or infinite and countable. If $A$ is a set, we denote by $[A]^{\leq\omega}$ the set of all its countable subsets. If $f$ is a mapping then we denote by $\rng f$ the range of $f$ and by $\dom f$ the domain of $f$. By writing $f:X\to Y$ we mean that $f$ is a mapping with $\dom f = X$ and $\rng f \subset Y$. By the symbol $f\restriction_{Z}$ we denote the restriction of the mapping $f$ to the set $Z$.

We shall consider Banach spaces over the field of real numbers (but many results hold for complex spaces as well). $B_X$ is the unit ball in $X$; i.e., the set $\{x\in X:\setsep \|x\| \leq 1\}$. $X^*$ stands for the (continuous) dual space of $X$. For a set $A\subset X^*$ we write $A_\bot = \{x\in X\setsep (\forall a\in A)\;a(x) = 0\}$. If $\{x_i\}_{i\in I}$ is a family of vectors in the Banach space $X$, we denote by $[x_i\setsep i\in I]$ the closed linear hull of $\{x_i\}_{i\in I}$. A set $D\subset X^*$ is \textit{r-norming} if, for every $x\in X$, $\|x\| \leq r. \sup\{|x^*(x)|:\;x^*\in D\cap B_{X^*}\}$.

\section{Rich families generated by suitable models}

In this section we give a method of generating rich families of separable subspaces with certain additional properties. The method is based on the proof of the famous L\"owenheim--Skolem Theorem; see e.g. \cite[Chapter IV Theorem 7.8]{kunen}. In Theorem \ref{tRich} we show that if $X$ is a Banach space with a fundamental minimal system or with $\dens X = \aleph_1$, then the method generates rich families of separable subspaces. However, it is not known to the authors whether the condition on a Banach space $X$ is necessary.

Let us first recall some definitions:

Let $N$ be a fixed set and $\varphi$ a formula in the language of $ZFC$. Then the {\em relativization of $\varphi$ to $N$} is the formula $\varphi^N$ which is obtained from $\varphi$ by replacing each quantifier of the form ``$\forall x$'' by ``$\forall x\in N$'' and each quantifier of the form ``$\exists x$'' by ``$\exists x\in N$''.

If $\varphi(x_1,\ldots,x_n)$ is a formula with all free variables shown (i.e., a formula whose free variables are exactly $x_1,\ldots,x_n$) then {\em $\varphi$ is absolute for $N$} if and only if
$$\forall a_1,\ldots,a_n\in N\quad (\varphi^N(a_1,\ldots,a_n) \leftrightarrow \varphi(a_1,\ldots,a_n)).$$

A list of formulas, $\varphi_1,\ldots,\varphi_n$, is said to be {\em subformula closed} if and only if every subformula of a formula in the list is also contained in the list.

Any formula in the set theory can be written using symbols $\in,=,\wedge,\vee,\neg,\rightarrow,\leftrightarrow,\exists,(,),[,]$ and symbols for variables. Let us assume a subformula closed list of formulas $\varphi_1,\ldots,\varphi_n$ is written in this way. Then it is not difficult to show, that the absoluteness of $\varphi_1,\ldots,\varphi_n$ for $N$ in other words says, that those formulas don't create any new sets in $N$. This result is contained in the following lemma (a proof can be found in \cite[Chapter IV, Lemma 7.3]{kunen}):

\begin{lemma}\label{lKunen}
 Let $N$ be a set and $\varphi_1,\ldots,\varphi_n$ subformula closed list of formulas (formulas containing only symbols $\in,=,\wedge,\vee,\neg,\rightarrow,\leftrightarrow,\exists,(,),[,]$ and symbols for variables). Then the following are equivalent:
  \begin{itemize}
    \item[(i)] $\varphi_1,\ldots,\varphi_n$ are absolute for $N$
    \item[(ii)] Whenever $\varphi_i$ is of the form $\exists x\varphi_j(x,y_1,\ldots y_l)$ (with all free variables shown), then
  \end{itemize}
  $$\quad\forall y_1,\ldots y_l\in N\left[\exists x\ (\varphi_j(x,y_1,\ldots y_l))\rightarrow(\exists x\in N)(\varphi_j(x,y_1,\ldots y_l))\right]$$
\end{lemma}

The method of creating rich families is based on the proof of the L\"owenheim--Skolem Theorem and on the following statement (for a proof see e.g. \cite[Chapter IV, Theorem 7.8]{kunen}).

\begin{thm}\label{tCountModel}
 Let $\varphi_1,\ldots,\varphi_n$ be any formulas and $Y$ any set. Then there exists a set $M\supset Y$ such, that
 $$(\varphi_1,\ldots,\varphi_n \text{ are absolute for }M)\quad \wedge\quad (|M|\leq \max(\omega,|Y|)).$$
\end{thm}

The set with the properties from previous theorem will be often used throughout the paper. Therefore we will use the following definition.

\begin{defin}
Let $\varphi_1,\ldots,\varphi_n$ be any formulas and let $Y$ be any countable set. Let $M\supset Y$ be a countable set satisfying that $\varphi_1,\ldots,\varphi_n$ are absolute for $M$. Then we say that {\em $M$ is a suitable model for $\varphi_1,\ldots,\varphi_n$ containing $Y$}. We denote this by $M\prec(\varphi_1,\ldots,\varphi_n;\; Y)$.

If $X$ is a topological space and $M\prec(\varphi_1,\ldots,\varphi_n;\; Y)$, we denote by $X_M$ the set $\ov{X\cap M}$.
\end{defin}

We will repeatedly use some results from \cite{cuth} which are summed up in the following lemma; see \cite[Proposition 2.9, 2.10 and 3.2]{cuth}.

\begin{lemma}\label{l:predp}
There are formulas $\theta_1,\dots,\theta_m$ and a countable set $Y_0$ such that any $M\prec(\theta_1,\ldots,\theta_n;\; Y_0)$ satisfies the following conditions:
\begin{itemize}
	\item If $f\in M$ is a mapping, then $\dom(f)\in M$, $\rng(f)\in M$ and $f[M]\subset M$.
	\item If $A\in M$ is a countable set, then $A\subset M$.
	\item If $\langle X,+,\cdot\rangle\in M$ is a vector space, then $X\cap M$ is a $\qe$-linear subspace of $X$.
\end{itemize}
\end{lemma}

In the sequel we will always assume that the formulas from the previous lemma are contained in the respective list of formulas and that the set $Y_0$ is contained in the respective countable set.

We will try to find out conditions under which suitable models generate nice rich families in the following sense.

\begin{defin}Let $X$ be a Banach space. We say that \emph{suitable models generate nice rich families in $X$}, if the following holds:
\begin{quotation}Whenever $Y$ is a countable set and $\varphi_1,\ldots,\varphi_n$ is a list of formulas, there exists a family $\M$ satisfying the following conditions:
\begin{enumerate}[\upshape (i)]
	\item For every $M\in\M$, $M\prec(\varphi_1,\ldots,\varphi_n;Y)$.
	\item The set $\{X_M\setsep M\in\M\}$ is a rich family of separable subspaces in $X$.
	\item $\forall M,N\in\M:\quad M\subset N \Longleftrightarrow \ov{X\cap M}\subset \ov{X\cap N}$.
\end{enumerate}\end{quotation}
\end{defin}

The main obstacle is to find a family $\M$ of suitable models such that $\{X_M\setsep M\in\M\}$ covers the space $X$ and the condition (iii) in the definition above is satisfied.

First, we recall the prescription on creating suitable models using fixed ``Skolem function'' (the prescription comes from the proof of the L\"owenheim--Skolem Theorem). Sets created with a fixed Skolem function will be used later in order to find the family $\M$ of suitable models.

Let us have a subformula closed list of formulas $\varphi_1,\ldots,\varphi_n$ and a set $R$ such that $\varphi_1,\ldots,\varphi_n$ are  absolute for $R$. Fix some well-ordering $\triangleleft$ on the set $R$. Now, for the given list of formulas $\varphi_1,\ldots,\varphi_n$, set $R$ and well-ordering $\triangleleft$, we define the \emph{Skolem function} $\psi$ \emph{for $\varphi_1,\ldots,\varphi_n$, $R$ and $\triangleleft$} in the following way:\\

\begin{itemize}
\item[1.] For every $i\in\{1,\ldots,n\}$, we denote by $l_i$ the number of all the free variables in the formula $\varphi_i$ and we define a mapping $H_i:R^{l_i}\to R$ in the following way:
\begin{itemize}
	\item if $l_i = 0$, then $R^{l_i} = \{\emptyset\}$ and $H_i(\emptyset)$ is the $\triangleleft$-least element of $R$.
	\item if $l_i > 0$ and $(r_1,\ldots,r_{l_i})\in R^{l_i}$ is fixed, then:
		\begin{itemize}
			\item[-] if $\varphi_i = \exists x\varphi_j(x,v_1,\ldots,v_{l_i})$ for some formula $\varphi_j$ and if there is some $x\in R$ such that $\varphi_j(x,r_1,\ldots,r_{l_i})$ holds, then  then $H_i(r_1,\ldots,r_{l_i})$ is the $\triangleleft$-least of such elements.		
			\item[-] in all the other cases, $H_i(r_1,\ldots,r_{l_i})$ is the $\triangleleft$-least element of $R$.
		\end{itemize}
\end{itemize}
\item[2.] Now, we define $\psi:[R]^{\leq\omega}\to [R]^{\leq\omega}$ in the following way:
\begin{itemize}
	\item Fix a countable set $A\subset R$ and put $A_0 = A$.
	\item Having $A_k$, we define $A_{k+1}$ by
$$A_{k+1} = A_k\cup \bigcup\{H_i(a_1,\ldots,a_{l_i})\setsep i = 1,\ldots,n\;,(a_1,\ldots,a_{l_i})\in (A_k)^{l_i}\}.$$
	\item Now, we define $\psi(A) = \bigcup_{k\in\omega}A_k$.
\end{itemize}
\end{itemize}

\begin{lemma}\label{lBasicSkolem}
 Let $\psi$ be a Skolem function for $\varphi_1,\ldots,\varphi_n$, $R$ and $\triangleleft$. Then
 \begin{enumerate}[\upshape (i)]
	\item For every $A\in[R]^{\leq\omega}$, $\psi(A)$ is a countable set, $\varphi_1,\ldots,\varphi_n$ are absolute for $\psi(A)$ and $\psi(A)\supset A$.
	\item The mapping $\psi$ is idempotent; i.e. $\psi\circ \psi = \psi$.
	\item The mapping $\psi$ is monotone; i.e. $\psi(A)\subset \psi(B)$ whenever $A, B\in[R]^{\leq\omega}$ are such that $A\subset B$.
	\item For every countable family $\F$ consisting of sets from $[R]^{\leq\omega}$,	$\psi(\bigcup\F) = \bigcup_{F\in\F}\psi(F)$.
	\item Let $J\subset R$ be an arbitrary set. Then, for every $A\in[J]^{\leq\omega}$ and $B\in[R]^{\leq\omega}$,
	$$\psi(A)\cap J\subset \psi(B) \Longleftrightarrow \psi(A)\subset \psi(B).$$

\end{enumerate}
\end{lemma}
\begin{proof}It follows immediately from the definition of the Skolem function $\psi$ that (ii) and (iii) holds. By Lemma \ref{lKunen} and the definition of the Skolem function $\psi$, (i) holds.

Let $\F$ be as in (iv). Then it follows from (iii) that $\psi(\bigcup\F) \supset \bigcup_{F\in\F}\psi(F)$. In order to prove the other inclusion, notice that $\bigcup\F\subset \bigcup_{F\in\F}\psi(F)$. By (iii) and (ii) respectively, $\psi(\bigcup\F)\subset \psi(\bigcup_{F\in\F}\psi(F)) = \bigcup_{F\in\F}\psi(\psi(F)) = \bigcup_{F\in\F}\psi(F)$. Thus, (iv) holds.

Let $J,A,B$ be as in (v). Then the implication from the right to the left is obvious. In order to prove the opposite one, let us suppose that $\psi(A)\cap J\subset \psi(B)$. By (i), $\psi(A)\cap J\supset A\cap J = A$. Hence, $\psi(B)\supset A$. Using (ii) and (iii), $\psi(B)\supset\psi(A)$.
\end{proof}

Under certain condition, the submodels created from the Skolem function above generate rich family of separable subspaces. The condition is contained in the following lemma. However, it is open to the authors whether this condition is necessary; see Question \ref{q1} bellow Lemma \ref{lRich}. We start with a definition.

\begin{defin}Let $\{x_i\}_{i\in I}$ be a family of vectors in the Banach space $X$.
The family of vectors $\{x_i\}_{i\in I}$ is said to be a \emph{fundamental minimal} system if $[x_i\setsep i\in I] = X$ and, for every $i\in I$, $x_i\notin [x_j\setsep j\in I, j\neq i]$.
\end{defin}

\begin{lemma}\label{lRich}Let $X$ be a Banach space. Let $\{x_i\}_{i\in I}$ be a system of vectors such that $[x_i\setsep i\in I] = X$ and let there exist a countable set $Y$ and formulas $\varphi_1,\ldots,\varphi_n$ such that for every $M\prec(\varphi_1,\ldots,\varphi_n;Y)$,
$$x_i\in[x_j\setsep j\in M\cap I]\Rightarrow i\in M.$$
Then suitable models generate nice rich families in $X$.
\end{lemma}
\begin{proof} Let us define a mapping $f:I\to X$ by $f(i) = x_i$, $i\in I$. Without loss of generality we may suppose that $f\in Y$ and $[I]^{\le\omega}\in Y$. Suppose furhter that $Y$ contains the set $Y_0$ from Lemma~\ref{l:predp} and the list of formulas $\varphi_1,\dots,\varphi_n$ contains the formulas from Lemma~\ref{l:predp} and the formulas marked below by $(*)$ and their subformulas.

We will show that if $M\prec(\varphi_1,\ldots,\varphi_n;Y)$, then $X_M$ is a separable subspace of $X$ and $X_M = [x_i\setsep i\in M\cap I]$.

Indeed, let $M\prec(\varphi_1,\ldots,\varphi_n;Y)$ be arbitrary. By Lemma~\ref{l:predp} we have $f[I\cap M]\subset X\cap M$ and $X\cap M$ is a $\qe$-linear set. Thus, $X_M$ is a separable subspace of $X$ and $[x_i\setsep i\in M\cap I] \subset X_M$. 
Further, for any $x\in X\cap M$ we have
$$\exists A\in [I]^{\le\omega} : x\in [x_i\setsep i\in A]\eqno{(*)}$$
By absoluteness of this formula there is such an $A$ in $M$. By Lemma~\ref{l:predp} we have $A\subset M$. Hence, $X\cap M\subset [x_i\setsep i\in I\cap M]$ and $X_M\subset [x_i\setsep i\in I\cap M]$.

In order to verify that suitable models generate nice rich families in $X$, fix a countable set $Z$ and formulas $\phi_1,\ldots,\phi_k$. By adding countably many sets to $Z$ and by adding finitely many formulas to $\phi_1,\ldots,\phi_k$, we may assume that $Z\supset Y$ and all the formulas $\varphi_1,\ldots,\varphi_n$ are contained in $\phi_1,\ldots,\phi_k$.

By Theorem \ref{tCountModel}, there exists a set $R\supset Z\cup\P(X)$ such that $\phi_1,...,\phi_k$ are absolute for $R$. Fix a well-ordering $\triangleleft$ on $R$ and let $\psi$ be the Skolem function  for $\phi_1,\ldots,\phi_k$, $R$ and $\triangleleft$. Put 
$$\M:=\{\psi(A\cup Z)\setsep A\in[I]^{\leq\omega}\}.$$
By Lemma \ref{lBasicSkolem} 
(i), for every 
$M\in\M$, 
$M\prec(\phi_1,\ldots,\phi_k;Z)$. We claim that
$$\F = \{X_M\setsep M\in\M\}$$
is the rich family of subspaces we are looking for.

It is easy to see that for every separable subspace $T\subset X$ there exists $M\in\M$ such that $T\subset X_M$ (it is enough to take any countable set $D$ such that $[x_i\setsep i\in D]\supset T$ and put $M = \psi(D\cup Z)$).

Fix $M,N\in\M$. Obviously, $X_M\subset X_N$ whenever $M\subset N$. On the other hand, let us assume that $X_M\subset X_N$; hence, $[x_i\setsep i\in I\cap M]\subset [x_i\setsep i\in I\cap N]$. By the assumption, $M\cap I\subset I\cap N$. By Lemma \ref{lBasicSkolem} (v), $M\subset N$. Hence, $X_M\subset X_N$ if and only if $M\subset N$.

Let us have a chain of spaces from $\F$, $X_{M_1}\subset X_{M_2}\subset\ldots$ where $M_n\in\M$ for every $n\in\en$. By the above, $M_1\subset M_2\subset\ldots$. By Lemma \ref{lBasicSkolem} (iv), $\bigcup_{n\in\en}M_n\in\M$. Now, it is easy to verify (see e.g. \cite[Lemma 3.4]{cuth}) that $\ov{\bigcup_{i=1}^\infty X_{M_i}} = X_{\bigcup_{n\in\en}M_n}\in\F$. Thus, $\F$ is a rich family of separable subspaces. This finishes the proof.
\end{proof}

\begin{question}\label{q1}Let $X$ be a Banach space. In order to verify that ``suitable models generate nice rich families in $X$'', it is sufficient to verify that the system of vectors $\{x_i\}_{i\in I}$ from Lemma \ref{lRich} exists. Is this condition also necessary?
\end{question}

In the following we summarize two important cases when condition of Lemma \ref{lRich} is satisfied. We do not know an example of a Banach space $X$ where suitable models do not generate rich families in $X$.

\begin{thm}\label{tRich}Let $X$ be a Banach space. Let $X$ have a fundamental minimal system $\{x_i\}_{i\in I}$ or $\dens X = \aleph_1$. Then suitable models generate nice rich families in $X$.
\end{thm}
\begin{proof}It is enough to verify that in both cases the assumptions of Lemma \ref{lRich} are satisfied.

First, let us assume that $\{x_i\}_{i\in I}$ is a fundamental minimal system. Let us fix an arbitrary set $J\subset I$. Then
$$i\notin J\Rightarrow x_i\notin[x_j\setsep j\in J].$$
Hence, the assumption of Lemma \ref{lRich} is obviously satisfied ($M$ need not be a suitable model but an arbitrary set).

Now, let us assume that $\dens X = \aleph_1$ and fix a family of vectors $\{x_\alpha\}_{\alpha<\omega_1}$ such that $[x_\alpha\setsep \alpha<\omega_1] = X$ and $x_\alpha\notin[x_\beta\setsep \beta<\alpha]$ for every $\alpha < \aleph_1$. Such a system can be easily constructed by transfinite induction.

Let us fix a list of formulas $\varphi_1,\dots,\varphi_n$ which contains the formulas from Lemma~\ref{l:predp} and a countable set $Y$ containing the set $Y_0$ from Lemma~\ref{l:predp} and the mapping $\alpha\mapsto x_\alpha$.
Let $M\prec(\varphi_1,\ldots,\varphi_n;Y)$ be arbitrary.

Then $\omega_1\in M$ (as the domain of the above mentioned mapping due to Lemma~\ref{l:predp}) and, moreover, for each $\alpha\in \omega_1\cap M$ we have $\alpha\subset M$ (again by Lemma~\ref{l:predp}). Set $\gamma=\sup \omega_1\cap M$.
Then in fact $\gamma=\omega_1\cap M$. Therefore for any $\alpha\in\omega_1$ we have
$$x_\alpha\in [x_\beta:\beta\in\omega_1\cap M] \Leftrightarrow  x_\alpha\in [x_\beta:\beta<\gamma]
\Leftrightarrow \alpha<\gamma\Leftrightarrow \alpha\in M.$$
Indeed, the first and the third equivalences follow from the above proven fact $\gamma=\omega_1\cap M$ and the second one follows from the properties of the family  $\{x_\alpha\}_{\alpha<\omega_1}$.
This verifies the assumption of Lemma~\ref{lRich}, so the proof is completed.\end{proof}

Let us recall that it is undecidable in ZFC whether every Banach space with $\dens X = \aleph_1$ has a fundamental minimal system; see e.g. \cite[Section 4.4]{hajek} and \cite[Corollary 6]{todo}. We know only one example of a Banach space with $\dens X > \aleph_1$ and without a fundamental minimal system. It is the space $\ell_\infty^c(\Gamma)$ for an index set $\Gamma$ of cardinality greater than continuum (we denote by $\ell_\infty^c(\Gamma)$ the closed subspace of $\ell_\infty(\Gamma)$ consisting of vectors of countable support); see e.g. \cite[Theorem 4.26]{hajek}. However, even in this case it is not clear to the authors whether suitable models generate nice rich families. Hence, the following question seems to be open.

\begin{question}\label{q2}
Do suitable models generate nice rich families in every Banach space? Do suitable models generate nice rich families in $\ell_\infty^c(\Gamma)$?
\end{question}

\section{Separable reduction theorems}

In this section we investigate the relationship between the ``method of suitable models'' and the ``method of rich families'' in separable determination theorems. In order to formulate our results as theorems we need the following precise definitions.
They are necessary to substitute metamathematical notions ``property'' or ``statement''. They do not cover all the conceivable situations but they cover the important cases where separable reductions were investigated.

\begin{defin} Let $\phi(X,A,f,y_1,\ldots,y_k)$ be a statement concerning the Banach space $X$, the set $A$ and the function $f$. More precisely, we consider  $\phi(X,+,\cdot,\|\cdot\|,A,f,y_1,\ldots,y_k)$, where
\begin{equation*}\begin{split}
\phi(X,+,\cdot,\|\cdot\|,A,f,y_1,\ldots,y_k) = &\; \langle X,+,\cdot,\|\cdot\|\rangle\text{ is a Banach space}, A\subset X, f\text{ is a function,}\\ & \dom(f)\subset X\text{ and }\tilde{\phi}(X,+,\cdot,\|\cdot\|,A,f,y_1,\ldots,y_k)
\end{split}\end{equation*}
for a formula $\tilde{\phi}(u_1,u_2,u_3,u_4,v,w,y_1,\ldots,y_k)$ with all free variables shown.

Let us fix some constants $C_1,\ldots,C_k$, a Banach space $X$, $A\subset X$ and  a function $f$ with $\dom(f)\subset X$. We say the statement \emph{$\phi(X,A,f,C_1,\ldots,C_k)$ is separably determined by the method of rich families}, if there exists a rich family $\F$ of separable subspaces of $X$ such that, for every $F\in\F$,
$$\phi(X,A,f,C_1,\ldots,C_k)\Longleftrightarrow \phi(F,A\cap F,f\restriction_{F},C_1,\ldots,C_k).$$

We say the statement \emph{$\phi(X,A,f,C_1,\ldots,C_k)$ is separably determined by the method of suitable models}, if there exists a countable set $Y$ and a finite list of formulas $\varphi_1,\ldots,\varphi_n$ such that whenever $M\prec(\varphi_1,\ldots,\varphi_n;Y)$ and $\{X,+,\cdot,\|\cdot\|,A,f,C_1,\ldots,C_k\}\subset M$, then
$$\phi(X,A,f,C_1,\ldots,C_k)\Longleftrightarrow \phi(X_M,A\cap X_M,f\restriction_{X_M},C_1,\ldots,C_k).$$
\end{defin}

In order to see that statements separably determined by the method of rich families are separably determined by the method of suitable models, we need the following.

\begin{proposition}\label{pRichImpliesModel}There exists a list of formulas $\varphi_1,\ldots,\varphi_n$ and a countable set $Y$ such that for every $M\prec(\varphi_1,\ldots,\varphi_n;\;Y)$ the following holds: Let $\langle X,+,\cdot,\|\cdot\|\rangle$ be a Banach space and $\F$ a rich family of separable subspaces of $X$. Then whenever $\{X,+,\cdot,\|\cdot\|,\F\}\subset M$, $X_M\in\F$.
\end{proposition}
\begin{proof} Fix a subformula-closed list of formulas $\varphi_1,\ldots,\varphi_n$ containing formulas from Lemma~\ref{l:predp}, the axiom of power set and the formulas marked below by $(*)$. Fix a countable set $Y$ containing the set $Y_0$ from Lemma~\ref{l:predp}.
 Fix $M\prec(\varphi_1,\ldots,\varphi_n;\;Y)$, $X$ and $\F$ as above. Assume that $\{X,+,\cdot,\|\cdot\|,\F\}\subset M$.

First, we will show that
\begin{equation}\label{eq:x_m}X_M = \ov{\bigcup\{F\setsep F\in\F\cap M\}}.\end{equation}
Indeed, let us fix $x\in X\cap M$. Then, using the absoluteness of the formula (and its subformula)
$$\exists F\quad (F\in \F\;\wedge\;x\in F),\eqno{(*)}$$
there exists $F\in\F\cap M$ such that $x\in F$. Hence,
$$X_M = \ov{X\cap M}\subset\ov{\bigcup\{F\setsep F\in\F\cap M\}}.$$
In order to see the opposite inclusion holds, let us fix $F\in\F\cap M$. Using the absoluteness of the formula (and its subformula)
$$\exists D\quad (D\text{ is a countable dense set in }F),\eqno{(*)}$$
there exists a countable set $D\in M$ such that $\ov{D} = F$. Hence, $D\subset M$ and $F = \ov{D}\subset X_M$. Consequently, ~\eqref{eq:x_m} holds.

Now, $\F\cap M$ is a nonempty, up-directed set. Indeed, $\F\cap M\neq\emptyset$ follows from the absoluteness of the formula (and its subformula)
$$\exists F\quad (F\in\F).\eqno{(*)}$$
Let us fix $F, G\in\F\cap M$. By the absoluteness of the formula (and its subformula)
$$\exists H\quad (H\in\F\;\wedge\;F\cup G\subset H),\eqno{(*)}$$
there exists $H\in\F\cap M$ such that $F\cup G\subset H$. Thus, $\F\cap M$ is an up-directed set.

Hence, there is an increasing sequence of sets $\{F_n\}_{n\in\en}$ from $\F\cap M$ such that $\bigcup \F\cap M = \bigcup_{n\in\en}F_n$. Consequently,
$$X_M = \ov{\bigcup\{F\setsep F\in\F\cap M\}} = \ov{\bigcup_{n\in\en}F_n}\in\F.$$
\end{proof}

Now we are ready to formulate and prove the main result.

\begin{thm}\label{tRichAModel}Let $X$ be a Banach space, $A\subset X$ and $f$ a function with $\dom(f)\subset X$. Let $\phi(X,A,f,C_1,\ldots,C_k)$ be a statement concerning the Banach space $X$, the set $A$, the function $f$ and constants $C_1,\ldots,C_k$. Consider the following conditions
\begin{enumerate}[\upshape (i)]
	\item $\phi(X,A,f,C_1,\ldots,C_k)$ is separably determined by the method of suitable models.
	\item $\phi(X,A,f,C_1,\ldots,C_k)$ is separably determined by the method of rich families.
\end{enumerate}
Then (ii) implies (i). Moreover, if $X$ have a fundamental minimal system $\{x_i\}_{i\in I}$ or $\dens X = \aleph_1$, then both conditions are equivalent.
\end{thm}
\begin{proof}Let us assume that (ii) holds. Let us fix the countable set $Y$ and the list of formulas $\varphi_1,\ldots,\varphi_n$ from Proposition \ref{pRichImpliesModel} and add to them the following formula and its subformulas
\begin{equation}\label{eq:family}\begin{split}\exists\F\quad (\F & \text{ is a rich family of separable subspaces of $X$ such that, for every $F\in\F$, }\\ & \phi(X,A,f,C_1,\ldots,C_k)\Longleftrightarrow \phi(F,A\cap F,f\restriction_{F},C_1,\ldots,C_k)).\end{split}\end{equation}
Denote such an extended list of formulas by $\phi_1,\ldots,\phi_k$. Fix $M\prec(\phi_1,\ldots,\phi_k;Y)$ with $\{X,+,\cdot,\|\cdot\|,A,f,C_1,\ldots,C_k\}\subset M$. By (ii) and the absoluteness of the formula ~\eqref{eq:family} (and its subformula), there exists a rich family $\F\in M$ of separable subspaces of $X$ such that, for every $F\in\F$,
$$\phi(X,A,f,C_1,\ldots,C_k)\Longleftrightarrow \phi(F,A\cap F,f\restriction_{F},C_1,\ldots,C_k).$$
By Proposition \ref{pRichImpliesModel}, $X_M\in\F$. Hence,
$$\phi(X,A,f,C_1,\ldots,C_k)\Longleftrightarrow \phi(X_M,A\cap X_M,f\restriction_{X_M},C_1,\ldots,C_k).$$
As $M$ was an arbitrary set with $M\prec(\phi_1,\ldots,\phi_k;Y)$ and $\{X,+,\cdot,\|\cdot\|,A,f,C_1,\ldots,C_k\}\subset M$, (i) holds.

Now, let us prove (i)$\Rightarrow$(ii) in the ``moreover'' part. By (i), there is a countable set $Z = Y\cup\{X,+,\cdot,\|\cdot\|,A,f,C_1,\ldots,C_k\}$ and a finite list of formulas $\varphi_1,\ldots,\varphi_n$ such that, for every $M\prec(\varphi_1,\ldots,\varphi_n;Z)$,
$$\phi(X,A,f,C_1,\ldots,C_k)\Longleftrightarrow \phi(X_M,A\cap X_M,f\restriction_{X_M},C_1,\ldots,C_k).$$
By Theorem \ref{tRich}, there exists a family $\M$ such that, for every $M\in\M$, $M\prec(\varphi_1,\ldots,\varphi_n;Z)$ and  $\F = \{X_M\setsep M\in\M\}$ is a rich family of separable subspaces in $X$. Hence, for every $F\in\F$,
$$\phi(X,A,f,C_1,\ldots,C_k)\Longleftrightarrow \phi(F,A\cap F,f\restriction_F,C_1,\ldots,C_k).$$
Consequently, (ii) holds.
\end{proof}

Let us show some consequences of the above. First, applying Theorem \ref{tRichAModel} on a result from \cite{cuthrmoutil}, we get the following; see e.g. \cite{cuthrmoutil} for the definition of a $\sigma$-upper porous set.

\begin{cor}\label{cPorous}Let $X, Y$ be Banach spaces and $f:X\to Y$ be a function. Let $X$ have a fundamental minimal system or $\dens X = \aleph_1$. Then there exists a rich family $\F$ of separable spaces of $X$ such that for every $F\in\F$ the following two conditions are equivalent:
\begin{enumerate}[(i)] 
  \item the set of the points where $f$ is not Fr\'echet differentiable is $\sigma$-upper porous,
  \item the set of the points where $f\restriction_F$ is not Fr\'echet differentiable is $\sigma$-upper porous in $F$.
\end{enumerate}
\end{cor}
\begin{proof}
Let us denote by $\phi(X,f,Y)$ the formula ``the set of the points where $f$ is not Fr\'echet differentiable is $\sigma$-upper porous in $X$''. By \cite{cuthrmoutil} (see the proof of Theorem 1.2 which is given just bellow the proof of Theorem 5.4), $\phi(X,f,Y)$ is separably determined by the method of suitable models. By Theorem \ref{tRichAModel}, $\phi(X,f,Y)$ is separably determined by the method of rich families.
\end{proof}

Applying Theorem \ref{tRichAModel} on a result from \cite{gakub}, we get the following result concerning nonseparable Gurari\u{\i} spaces. A Banach space $X$ is said to be a Gurari\u{\i} space if, for every pair of finite-dimensional spaces $S\subset T$, for every isometric embedding $f:S\to X$ and for every $\eps > 0$, there exists
an $\eps$-isometric embedding $g:T\to X$ such that $g\restriction_S = f$. Let us recall there exists exactly one (up to isometry) separable Gurari\u{\i} space. For a survey on Gurari\u{\i} spaces, see e.g. \cite{gakub}.

\begin{cor}There exists a list of formulas $\phi_1,\ldots,\phi_n$ and a countable set $Y$ such that for every $M\prec(\phi_1,\ldots,\phi_n;\;Y)$ the following holds: Let $X$ be a Banach space with $\{X,+,\cdot,\|\cdot\|\}\subset M$. Then $X$ is a Gurari\u{\i} space if and only if $X_M$ is a Gurari\u{\i} space.
\end{cor}
\begin{proof}Let us denote by $\phi(X)$ the formula ``$X$ is a Gurari\u{\i} space''. By \cite[Theorem 3.4]{gakub}, $\phi(X)$ is separably determined by the method of rich families. By Theorem \ref{tRichAModel}, $\phi(X)$ is separably determined by the method of suitable models.
\end{proof}

\begin{rem}Let $X$ be a Banach space, $A\subset X$, $f$ a function with domain in $X$ and $\phi(X,A,f)$ a statement determined by the method of suitable models. It is known to the authors that if $\phi(X,A,f)$ holds, then there does not exist a rich family of separable subspaces $\F$ such that, for every $F\in\F$, $\neg\phi(F,A\cap F,f\restriction_F)$ holds. The argument is as follows. Arguing by contradiction, let us fix such a rich family of separable subspaces $\F$. By the assumption, there is a countable set $Y$ and a finite list of formulas $\varphi_1\ldots,\varphi_n$ such that, whenever $M\prec(\varphi_1\ldots,\varphi_n;Y)$, $\phi(X_M,A\cap X_M,f\restriction_{X_M})$ holds. We inductively find sequences $(F_n)_{n\in\en}$ and $(M_n)_{n\in\en}$ such that
\begin{itemize}
	\item for every $n\in\en$, $M_n\prec(\varphi_1\ldots,\varphi_n;Y)$ and $F_n\in\F$;
	\item $M_1\subset M_2\subset \ldots$ and $F_1\subset X_{M_1}\subset F_2\subset X_{M_2}\subset\ldots$.
\end{itemize}
Put $M = \bigcup_{n=1}^\infty M_n$. It is easy to check that $M\prec(\varphi_1\ldots,\varphi_n;Y)$ and $X_M = \ov{\bigcup_{n=1}^\infty X_{M_n}}$; see, e.g., \cite{cuth}. Thus, $X_M = \ov{\bigcup_{n=1}^\infty F_n}\in\F$ and $\phi(X_M,A\cap X_M,f\restriction_{X_M})$ holds. This is a contradiction.

However, the following question remains open.
\end{rem}

\begin{question}\label{q3}Is the method of suitable models equivalent to the method of rich families? More precisely, does (i)$\Rightarrow$(ii) hold in Theorem \ref{tRichAModel} if we do not assume that $X$ has a fundamental minimal system or is of density $\aleph_1$? Does it hold at least for $\C(K)$ spaces? Does it hold for $\C(K)$ spaces, where $K$ is a Boolean space?
\end{question}

\section{Projectional skeletons}

In the last section we apply the previous results to give a characterization of spaces with a projectional skeleton. We prove that a projectional skeleton may be without loss of generality considered to be simple in the sense of the following definition.

\begin{defin}A \textit{simple projectional skeleton} in a Banach space $X$ is a family of bounded projections $\{P_F\}_{F\in\F}$ indexed by a rich family of separable subspaces $\F$ satisfying the following conditions:
\begin{enumerate}[\upshape (i)]
	\item for every $F\in\F$, $P_F(X) = F$;
	\item if $E\subset F$ in $\F$, then  $P_E = P_E\circ P_F = P_F\circ P_E.$
\end{enumerate}
Given $r\geq 1$, we say that $\{P_F\}_{F\in\F}$ is an \textit{simple r-projectional skeleton} if it is a simple projectional skeleton such that $\|P_F\|\leq r$ for every $F\in\F$.

We say that $\{P_F\}_{F\in\F}$ is a \textit{simple commutative projectional skeleton} if $P_E\circ P_F = P_F\circ P_E$ for any $E,F\in\F$.
\end{defin}

\begin{defin}Let $\mathfrak{s} = \{P_s\}_{s\in\Gamma}$ be a (simple) projectional skeleton in a Banach space $X$ and let $D(\mathfrak{s}) = \bigcup_{s\in\Gamma}P^*_s[X^*]$. Then we say that \textit{$D(\mathfrak{s})$ is induced by a (simple) projectional skeleton}.\end{defin}

\begin{thm}\label{tSkeleton}Let $X$ be a Banach space and let $D\subset X^*$ be an $r$-norming subspace ($r\geq 1$). The following properties are equivalent:
\begin{enumerate}[\upshape (i)]
	\item $X$ has a (commutative) $r$-projectional skeleton $\mathfrak{s}$ with $D = D(\mathfrak{s})$.
	\item $X$ has a simple (commutative) $r$-projectional skeleton $\mathfrak{s}$ with $D = D(\mathfrak{s})$.
\end{enumerate}
\end{thm}

 As we have remarked above, it is claimed already in \cite{kubisSkeletonEkviv} that such a statement holds. However, the proof  contains a gap as it is not clear why $X_M\subset X_N$ should imply $M\subset N$ (for $M,N$ suitable models). Here, using the preceding results, we fill in the gap. Moreover, as this result could interest broader audience (not familiar with elementary submodels), we give a proof which is absolutely self-contained. Of course, the idea of both proofs is the same. In the second proof we avoid the use of suitable models by giving in fact the proof of Theorem~\ref{tCountModel} in the concrete case.

\begin{proof}[\textbf{Proof 1 - using the method of suitable models}]It is obvious that every simple projectional skeleton is a projectional skeleton. Thus, (i) easily follows from (ii).

Let us assume $\mathfrak{s} = \{P_s\}_{s\in\Gamma}$ is an $r$-projectional skeleton in $X$ and $D = D(\mathfrak{s})$. It follows immediately from the proof of \cite[Lemma 14]{kubis} that there exist a countable set $Y$ and a finite list of formulas $\varphi_1,\ldots,\varphi_n$ such that whenever $M\prec(\varphi_1,\ldots,\varphi_n;Y)$, there exists a projection $P_M$ with $\|P_M\|\leq r$, $P_M(X) = X_M$ and $\ker(P_M) = (D\cap M)_\bot$.

Recall, that every space with a projectional skeleton has a Markushevich basis. Every Markushevich basis is a fundamental minimal system (this is immediate from the definition, see \cite[Definition 1.7]{hajek}). Hence, by Theorem \ref{tRich}, there exists a family $\M$ such that
\begin{itemize}
	\item[$\bullet$] for every $M\in\M$, $M\prec(\varphi_1,\ldots,\varphi_n;Y)$,
	\item[$\bullet$] the set $\F = \{X_M\setsep M\in\M\}$ is a rich family of separable subspaces in $X$ and
	\item[$\bullet$] $\forall M,N\in\M:\quad M\subset N \Longleftrightarrow \ov{X\cap M}\subset \ov{X\cap N}$.
\end{itemize}
$\{P_M\}_{M\in\M}$ can be equivalently indexed as $\{P_M\}_{X_M\in\F}$. Let us fix $M,N\in\M$ with $X_M\subset X_N$. Then obviously $P_M(X)\subset P_N(X)$ and $P_M = P_N\circ P_M$. Moreover, $M\subset N$. Thus, $\ker(P_M) = (D\cap M)_\bot\supset (D\cap N)_\bot = \ker(P_N)$; hence, $P_M = P_M\circ P_N$. Consequently, $\mathfrak{s'} = \{P_M\}_{X_M\in\F}$ is a simple $r$-projectional skeleton in $X$. It is clear that $D(\mathfrak{s'}) = \bigcup_{M\in\M}P_M^*(X^*) = \bigcup_{M\in\M}\ov{D\cap M}^{w^*}$. As any set induced by a projectional skeleton is countably closed, we have that $D(\mathfrak{s'})\subset D$. By \cite[Corollary 19]{kubis}, $D(\mathfrak{s'}) = D$.

It follows immediately from the proof of \cite[Lemma 14]{kubis} that, for every $M\in\M$, projection $P_M$ equals $P_s$ for some $s\in\Gamma$ (more precisely, $s = \sup(\Gamma\cap M)$). Hence, $\mathfrak{s'}$ is commutative whenever $\mathfrak{s}$ is commutative.
\end{proof}
\begin{proof}[\textbf{Proof 2 - not using the method of suitable models}]It is obvious that every simple projectional skeleton is a projectional skeleton. Thus, (i) easily follows from (ii).

Let us assume $\mathfrak{s} = \{P_s\}_{s\in\Gamma}$ is an $r$-projectional skeleton in $X$ and $D = D(\mathfrak{s})$. Fix a Markushevich basis $\{x_i\}_{i\in I}$ on $X$.

For every $A\in[I]^{\leq\omega}$ we will find sets $I(A)\in[I]^{\leq\omega}$, $D(A)\in[D]^{\leq\omega}$ and an up-directed set $\Gamma(A)\in[\Gamma]^{\leq\omega}$ such that if we put $t_A = \sup(\Gamma(A))$ and $X_A = [x_i\setsep i\in I(A)]$ then
\begin{enumerate}[\upshape (a)]
	\item\label{iCorrect} for every $A, B\in[I]^{\leq\omega}$, if $X_A = X_B$ then $\Gamma(A) = \Gamma(B)$;
	\item\label{iRich1} for every $A\in[I]^{\leq\omega}$, $A\subset I(A)$;
	\item\label{iRich2} for every sequence of sets $(A_n)_{n=1}^\infty$ from $[I]^{\leq\omega}$, $X_{\bigcup_{n\in\en} A_n} = \ov{\bigcup_{n\in\en} X_{A_n}}$;
	\item\label{iProj} for every $A\in[I]^{\leq\omega}$, $\rng(P_{t_A}) = X_A$ and $\ker(P_{t_A}) = D(A)_\bot$;	
	\item\label{iKer} for every $A, B\in[I]^{\leq\omega}$, if $X_A\subset X_B$ then $D(A)\subset D(B)$.
\end{enumerate}
Let us first verify that $\F = \{X_A\setsep A\in[I]^{\leq\omega}\}$ is a rich family of separable subspaces and $\mathfrak{s}'=\{P_F\}_{F\in\F}$ is a simple $r$-projectional skeleton, where $P_{X_A} = P_{t_A}$ for every $X_A\in\F$ (by ~\eqref{iCorrect}, $P_{X_A}$ is well-defined), $D = D(\mathfrak{s}')$ and the skeleton $\mathfrak{s}'$ is commutative whenever $\mathfrak{s}$ is commutative.

Having a separable subspace $Y$ find $A\in[I]^{\leq\omega}$ such that $Y\subset [x_i\setsep i\in A]$. By ~\eqref{iRich1}, $Y\subset X_A$. It follows immediately from ~\eqref{iRich2} that, for every increasing sequence $F_i$ in $\F$, $\overline{\bigcup_{i=1}^\infty F_i}$ belongs to $\F$. Hence, $\F$ is a rich family of separable subspaces. By ~\eqref{iProj}, $\rng(P_F) = F$ for every $F\in\F$. Fix $E,F\in\F$ with $E\subset F$. Then $P_E = P_F\circ P_E$, because $\rng(P_E)\subset \rng(P_F)$. By ~\eqref{iKer} and ~\eqref{iProj}, $\ker(P_F)\subset\ker(P_E)$; thus, $P_E = P_E\circ P_F$. Hence, $\mathfrak{s}'$ is a simple projectional skeleton. For every $F\in\F$, $\|P_F\|\leq r$ and $\mathfrak{s}'$ is commutative whenever $\mathfrak{s}$ is commutative. Moreover, $D(\mathfrak{s}') = \bigcup_{A\in[I]^{\leq\omega}}\ov{D(A)}^{w^*}$. As any set induced by a projectional skeleton is countably closed, we have that $D(\mathfrak{s'})\subset D$. By \cite[Corollary 19]{kubis}, $D(\mathfrak{s'}) = D$.

It remains to construct, for every $A\in[I]^{\leq\omega}$, sets $I(A)\in[I]^{\leq\omega}$, $D(A)\in[D]^{\leq\omega}$ and an up-directed set $\Gamma(A)\in[\Gamma]^{\leq\omega}$ satisfying the conditions above. 
In the following we say a set $E\subset X^*$ is $r$-norming for $Z\subset X$ if, for every $x\in Z$, $\|x\|\leq r\sup\{x^*(x)\setsep x^*\in B_{X^*}\cap E\}$. 
Fix a well-ordering $\triangleleft$ on the set $D\cup \Gamma\cup [I]^{\leq\omega}$

Fix $A\in[I]^{\leq\omega}$. Now, we inductively define sequences of countable sets $(I_n^A)_{n=1}^\infty$, $(\Gamma_n^A)_{n=1}^\infty$, $(D_n^A)_{n=1}^\infty$ in the following way. We put $I_1^A = A$, $\Gamma_1^A = \emptyset$ and $D_1^A = \emptyset$. Let us assume that $(I_n^A)_{n=1}^k$, $(\Gamma_n^A)_{n=1}^k$ and $(D_n^A)_{n=1}^k$ were already defined. We put $X_k^\qe = \qe$-$\sspan\{x_i\setsep i\in I_k^A\}$. This a countable set. In order to satisfy ~\eqref{iProj} and ~\eqref{iRich1}, it is enough to define $I_{k+1}^A$, $\Gamma_{k+1}^A$ and $D_{k+1}^A$ in such a way that those sequences are increasing and
\begin{itemize}
	\item $D_{k+1}^A$ is $r$-norming for $X_k^\qe$,
	\item $\Gamma_{k+1}^A$ is an up-directed set satisfying $X_k^\qe\subset \bigcup_{s\in\Gamma_{k+1}^A}P_s(X)$ and $D_{k+1}^A\subset \bigcup_{s\in\Gamma_{k+1}^A}P_s^*(X^*)$,
	\item $I_{k+1}^A$ is a set satisfying $\bigcup_{s\in\Gamma_{k+1}^A}P_s(X)\subset [x_i\setsep i\in I_{k+1}^A]$.
\end{itemize}

However, in order to satisfy ~\eqref{iCorrect}, ~\eqref{iRich2} and ~\eqref{iKer}, this construction must be done in a more precise way. Now, we define
\begin{equation*}\begin{split}D_{k+1}^A &\; = D_k\cup \\ & \bigcup_{l\in\en}\bigcup_{x\in X_k^\qe}\{d\in D\setsep d\text{ is the $\triangleleft$-least element of $D$ with }\|d\|\leq 1\text{ and }\|x\|\leq r|d(x)| + 1/l\}.\end{split}\end{equation*}
Having defined $D_{k+1}^A$, we put $\Gamma_{k+1}^A = \bigcup_{m\in\en} \Gamma_{k+1,m}^A$ where $\Gamma_{k+1,m}^A$ are inductively defined as follows.
\begin{align*}
	\Gamma_{k+1,1}^A = \;& \Gamma_k\cup \bigcup_{x\in X_k^\qe}\{s\in\Gamma\setsep s \text{ is the $\triangleleft$-least element of $\Gamma$ satisfying }x\in P_s(X)\}\cup \\
	& \bigcup_{d\in D_{k+1}}\{s\in\Gamma\setsep s \text{ is the $\triangleleft$-least element of $\Gamma$ satisfying }d\in P_s^*(X^*)\},\\
	\Gamma_{k+1,m+1}^A = \;& \Gamma_{k+1,m}^A\cup\bigcup_{u,v\in\Gamma_{k+1,m}^A}\{s\in\Gamma\setsep s \text{ is the $\triangleleft$-least element of $\Gamma$ satisfying }s\geq u,v\}.	
\end{align*}
Finally, we put
\begin{equation*}\begin{split}I_{k+1}^A &\; = I_k^A\cup \\ & \bigcup_{s\in\Gamma_{k+1}^A}\{\cup S\setsep S\in[I]^{\leq\omega}\text{ is the $\triangleleft$-least element of $[I]^{\leq\omega}$ satisfying }P_s(X)\subset [x_i\setsep i\in S]\}.\end{split}\end{equation*}
Now $I(A) = \bigcup_{n\in\en}I_n^A$, $\Gamma(A) = \bigcup_{n\in\en}\Gamma_n^A$ and $D(A) = \bigcup_{n\in\en}D_n^A$. It follows immediately from the construction above that
\begin{enumerate}[\upshape (1)]
	\item $D(A)$ is $r$-norming for $X_A$,
	\item\label{iTwo} $\Gamma(A)$ is an up-directed set satisfying $X_A\subset \ov{\bigcup_{s\in\Gamma(A)}P_s(X)}$ and $D(A)\subset \bigcup_{s\in\Gamma(A)}P_s^*(X^*)$,
	\item\label{iThree} $\bigcup_{s\in\Gamma(A)}P_s(X)\subset [x_i\setsep i\in I(A)] = X_A$,
	\item\label{iFour} for every $A, B\in[I]^{\leq\omega}$, if $X_A\subset X_B$ then $I(A)\subset I(B)$, $D(A)\subset D(B)$ and $\Gamma(A)\subset \Gamma(B)$,
	\item\label{iFive} for every $A\in[I]^{\leq\omega}$, $A\subset I(A)$,
	\item\label{iSix} for every $A\in[I]^{\leq\omega}$, $I(A) = I(I(A))$,
	\item\label{iSeven} for every $A,B\in[I]^{\leq\omega}$ with $A\subset B$, $I(A)\subset I(B)$.
\end{enumerate}
Now, ~\eqref{iCorrect} and ~\eqref{iKer} follow from ~\eqref{iFour}; ~\eqref{iRich1} from ~\eqref{iFive}. From ~\eqref{iTwo} and ~\eqref{iThree} it follows that
$$P_{t_A}(X) = \ov{\bigcup_{s\in\Gamma(A)}P_s(X)} = X_A.$$
In order to see that $\ker(P_{t_A}) = D(A)_\bot$, let us fix $x\in\ker(P_{t_A})$ and $d\in D(A)$. By ~\eqref{iTwo}, there exists $s\in\Gamma(A)$ with $P_s^*(d) = d$; hence, $d(x) = P_s^*(d)(x) = d(P_s x) = d(P_s\circ P_{t_A} x) = d(0) = 0$ and $x\in D(A)_\bot$. Thus, $\ker(P_{t_A}) \subset D(A)_\bot$. Moreover, since $D_A$ is $r$-norming for $X_A = \rng(P_{t_A})$, $\rng(P_{t_A})\cap D(A)_\bot = \{0\}$. Consequently, $\ker(P_{t_A}) = D(A)_\bot$. We have verified that ~\eqref{iProj} holds.

It remains to verify ~\eqref{iRich2}. Let us fix a sequence of sets $(A_n)_{n=1}^\infty$ from $[I]^{\leq\omega}$. By ~\eqref{iSix} and ~\eqref{iSeven}, similarly as in the proof of Lemma \ref{lBasicSkolem} (iv), $I(\bigcup_n A_n) = \bigcup_{n=1}^\infty I(A_n)$. Hence,
$$X_{\bigcup_{n\in\en} A_n} = [x_i\setsep i\in I(\cup_{n\in\en} A_n)] = [x_i\setsep i\in \cup_{n\in\en} I(A_n)] = \ov{\cup_{n\in\en} [x_i\setsep i\in I(A_n)]} = \ov{\bigcup_{n\in\en} X_{A_n}}$$
and ~\eqref{iRich2} holds.
\end{proof}

\begin{rem}Inspecting the proof of Theorem \ref{tSkeleton}, it is immediate that we can without loss of generality assume the projections in a  projectional skeleton are all created from suitable models. More precisely, let $\varphi_1,\ldots,\varphi_n$ be a finite list of formulas and $Y$ a countable set. Let $X$ be a Banach space and let $D\subset X^*$ be an $r$-norming subspace ($r\geq 1$). Then the following assertions are equivalent:
\begin{enumerate}[\upshape (i)]
	\item $X$ has a (commutative) $r$-projectional skeleton $\mathfrak{s}$ with $D = D(\mathfrak{s})$.
	\item $X$ has a simple (commutative) $r$-projectional skeleton $\{P_F\}_{F\in\F}$ with $D = D(\mathfrak{s})$ such that, for every $F\in\F$, there exists $M\prec(\varphi_1,\ldots,\varphi_n;Y)$ satisfying $\rng P_F = X_M$ and $\ker P_F = (D\cap M)_\bot$.
\end{enumerate}
\end{rem}

\begin{rem}If $X$ has a commutative projectional skeleton, we may construct a simple commutative skeleton in a more transparent way. However, some nontrivial results are used. We refer reader to \cite{kubisKniha, kubisSmall, cuthSimul} where necessary definitions used bellow may be found. First, $X$ has a commutative 1-projectional skeleton under an equivalent renorming; hence, by \cite[Proposition 29]{kubis}, we may without loss of generality assume $(B_{X^*},w^*)$ has a commutative retractional skeleton. By \cite[Theorem 6.1]{kubisSmall}, $(B_{X^*},w^*)$ is a Valdivia compact space. For an arbitrary Valdivia compact $K$ it is easy to construct a simple commutative projectional skeleton in the space $\C(K)$.

Indeed, let us follow the ideas and notation used in the proof of \cite[Theorem 19.14]{kubisKniha}. Let $K\subset[0,1]^\kappa$ be such that $\Sigma(\kappa)\cap K$ is dense in $K$. We call a set $T\subset\kappa$ \emph{admissible} if $x\cdot\chi_T\in K$ for every $x\in K$ and we denote by $\Gamma$ the set of all countable admissible subsets of $\kappa$. For every $S\in\Gamma$, we define $r_S:K\to K$ by $r_S(x) = x\cdot\chi_{S}$, $x\in K$. Then $\{r_S\}_{S\in\Gamma}$ is a commutative retractional skeleton. Now, for every $S\in\Gamma$, we define $P_S:\C(K)\to\C(K)$ by $P_S(f) = f\circ r_S$, $f\in\C(K)$. Then $\mathfrak{s} = \{P_S\}_{S\in\Gamma}$ is a commutative projectional skeleton. Moreover, assuming (without loss of generality) that $S = T$ whenever $r_S = r_T$, it is not difficult to verify that $\rng(P_S)\subset \rng(P_T)$ if and only if $S\subset T$; hence, $\mathfrak{s}$ is a simple commutative projectional skeleton.

It follows that $\C(B_{X^*},w^*)$ has a simple projectional skeleton $\{P_S\}_{S\in\Gamma}$. By \cite[Proposition 3.1 and Lemma 4.4]{cuthSimul}, we may without loss of generality assume that $\{P_S\restriction_X\}_{S\in\Gamma}$ is a simple projectional skeleton in $X$.

We used nontrivial results \cite[Theorem 6.1]{kubisSmall} and \cite[Proposition 3.1]{cuthSimul}; however, the construction of the simple projectional skeleton itself comes from the Valdivia case which is easier to handle.
\end{rem}

\section*{Acknowledgements}

M. C\'uth was supported by the Grant No. 282511 of the Grant Agency of the Charles University in Prague, O. F. K.  Kalenda by the grant GA \v{C}R P201/12/0290.

\end{document}